\documentclass[10pt,a4paper,reqno]{article}
\usepackage[T1]{fontenc}
\usepackage[T1]{fontenc}
\usepackage[all]{xy}
\usepackage{epsfig}
\usepackage{amssymb}
\usepackage{amsmath}
\usepackage{graphics,graphicx}
\usepackage{eepic}

\usepackage{vmargin}
\usepackage{theorem}
\usepackage{amsbsy}
\usepackage{amsfonts}
\title{Counterexamples to Symmetry\\
for Partially Overdetermined Elliptic Problems}
\author{Ilaria Fragal\`a, Filippo Gazzola, Jimmy Lamboley, Michel Pierre}
\date{July 2008}
\usepackage{epsfig}

\def\R{\mathbb{R}}

\newenvironment{summary}{\vskip\baselineskip \noindent\small\bf Summary: \rm}%
{\vskip\baselineskip}
\newenvironment{proof}{{\vskip\baselineskip\noindent\textbf{Proof:}}}%
{\hspace*{.1pt}\hspace*{\fill}\BOX\vskip\baselineskip}

\newcommand{\BOX}{\ensuremath\Box}
\newtheorem{theorem}{Theorem }[section]

{\theorembodyfont{\rmfamily}}
{\theorembodyfont{\rmfamily}}
{\theorembodyfont{\rmfamily}}

\newtheorem{proposition}[theorem]{Proposition}
{\theorembodyfont{\rmfamily}\newtheorem{remark}[theorem]{Remark}}
{\theorembodyfont{\rmfamily}}
\newenvironment{proofx}[1]%
{\vskip\baselineskip\noindent\textbf{Proof of {#1}:}}%
{\hspace*{.1pt}\hspace*{\fill}\BOX\vskip\baselineskip}
{\vskip\baselineskip\noindent\textbf{Proof of Theorem \protect\ref{#1}:}}%
{\hspace*{.1pt}\hspace*{\fill}\BOX\vskip\baselineskip}
{\vskip\baselineskip\noindent\textbf{Proof of Theorems \protect\ref{#1} --
\protect\ref{#2}:}}%
{\hspace*{.1pt}\hspace*{\fill}\BOX\vskip\baselineskip}

\begin{document}
\maketitle\thispagestyle{empty}


\begin{summary}
We exhibit several counterexamples showing that the famous Serrin's symmetry result for semilinear elliptic overdetermined problems may not hold for {\em partially} overdetermined problems, that is when both Dirichlet and Neumann boundary conditions are prescribed only {\em on part} of the boundary. Our counterexamples enlighten subsequent positive symmetry results obtained by the first two authors for such partially overdetermined
systems and justify their assumptions as well.
\end{summary}

\renewcommand{\thefootnote}{}
\footnotetext{\hspace*{-.51cm}AMS 1991 subject classification: Primary: 35J70; Secondary: 35B50,49Q10\\ %
}

\section{Introduction}\label{sect:intro}

Let $\Omega$ be an open bounded connected subset of $\R^n$ with smooth enough boundary, and let $\Gamma$ be a nonempty connected (relatively) open subset
of $\partial\Omega$. Let also $\nu$ denote the unit outer normal to $\partial\Omega$, $c$ be a positive constant and $f:\R\to\R$ be a smooth function.
By {\it ``overdetermined problem''}, we mean any boundary value problem of the following kind:
\begin{equation}\label{eq:over1}\left\{\begin{array}{ll}
-\Delta u=f(u)\qquad&\hbox{ in }\Omega\\
u=0 \quad \hbox{ and } \quad u_\nu =-c \qquad&\hbox{ on } \Gamma \\
u=0 \qquad & \hbox{ on } \partial \Omega \setminus \Gamma\ ,
\end{array}\right.
\end{equation}
or
\begin{equation}\label{eq:over2}
\left\{\begin{array}{ll}
-\Delta u=f(u)\qquad&\hbox{ in }\Omega\\
u=0 \quad \hbox{ and } \quad u_\nu =-c \qquad&\hbox{ on } \Gamma \\
|\nabla u| = c \qquad &\hbox{ on } \partial \Omega \setminus
\Gamma\ ,
\end{array}\right.
\end{equation}
where $u_\nu$ denotes the normal derivative of $u$ on $\partial\Omega$. Here
and in the sequel, by a solution $u$ to problem (\ref{eq:over1}) (resp. (\ref{eq:over2})), we always mean that
$u\in\mathcal{C}^0(\overline{\Omega})\cap\mathcal{C}^1(\Omega\cup\Gamma)\cap \mathcal{C}^2(\Omega)$ (resp. $\mathcal{C}^1(\overline{\Omega})\cap \mathcal{C}^2(\Omega)$).

The choice of the word ``overdetermined'' is justified by the presence of both the Dirichlet and Neumann conditions on a same nonempty part $\Gamma$ of the boundary in problems
(\ref{eq:over1})-(\ref{eq:over2}): this makes them in general not well-posed. Thus the existence of a solution to (\ref{eq:over1})
or (\ref{eq:over2}) is not always guaranteed, and, if existence happens to hold,  it is actually supposed to imply some severe geometric constraint on  $\Omega$.\par
This kind of problem was studied by Serrin \cite{Serrin}. His celebrated result states that, in the case of
{\it totally} overdetermined problems, that is when $\Gamma \equiv\partial\Omega$, then existence of a solution implies that
$\Omega$ is a ball (and $u$ is radially symmetric).

More recently, the case of {\it partially} overdetermined problems, that is when $\Gamma\varsubsetneq\partial\Omega$, has been studied by
the first two authors in \cite{FG}, where they investigate the following natural question:
\\
``If $\Gamma \varsubsetneq\partial\Omega$, can we still conclude that $\Omega$ is a ball\\
\hspace*{3cm}whenever $(\ref{eq:over1})$ or $(\ref{eq:over2})$ admits a solution?''

The answer is trivially {\bf \em no} without any extra natural geometric restriction on $\Omega$.
Assume, for instance, that $\Omega$ is an annulus, that is $\Omega=\{x\in\R^n; 0<a<|x|<b\}$. Then, the solution of $-\Delta u=1$ on $\Omega$, with $u=0$ on its boundary, is radially symmetric. Therefore, $u_\nu$ is equal to a constant on each piece of the boundary, but with different constants for each of them.

On the other hand, if $\partial\Omega$ is assumed to be connected, the problem becomes much more significant and delicate. In fact there are many
different situations where the answer to the above question is {\bf \em yes}, so that Serrin's symmetry result continues to hold.
This occurs under suitable additional assumptions, involving both regularity and geometric features, on the source term $f$ and the
overdetermined region $\Gamma$: for the detailed statements, as well as for a more extensive bibliography about overdetermined
problems, we refer to \cite{FG}.

The goal of this note is to show that there are nontrivial cases (meaning in particular that  $\partial\Omega$ is connected)
when the requirements of \cite{FG} are not satisfied and problems like (\ref{eq:over1})-(\ref{eq:over2}) admit a solution in domains
$\Omega$ different from a ball.

The counterexamples we construct for problems of type (\ref{eq:over1}) or (\ref{eq:over2}) are of different kind. Problems of type
(\ref{eq:over1}) are treated in Section \ref{sect:opti} by an approach based on shape optimization and domain derivative.
More precisely, we consider the problem of minimizing the Dirichlet energy of domains
with prescribed volume and confined in a planar box, that is
\begin{equation}\label{eq:opti}
|\Omega^*|=\alpha,\;\Omega^*\subset D,\;J(\Omega^*)=\min_{|\Omega|=\alpha, \Omega\subset D} J(\Omega),
\end{equation}
where $D=(-1, 1) ^2$ and
\begin{equation}\label{JJ}
J(\Omega):=\inf_{v\in H^1_0(\Omega)}\left\{\int_\Omega\left(\frac{1}{2}|\nabla v|^2-v\right)dx\right\}.
\end{equation}
Choosing $\alpha$ in a suitable range and applying the regularity results in \cite{B, BHP}, we obtain that (\ref{eq:opti})
admits an optimal open shape $\Omega^*$ with a nonempty smooth ``free boundary'' $\partial \Omega^* \cap D$. Then, writing down the optimality conditions by using shape derivatives, we are lead to a problem of type (\ref{eq:over1}) on $\Omega^*$, with $f\equiv1$ and $\Gamma=\partial\Omega^*\cap D$.

Problems of type (\ref{eq:over2}) are treated in Section \ref{sect:expl} by a different approach, which works in any dimension $n\ge2$.
In this case, the counterexamples are derived through some explicit computations. They are based on the idea of studying the zero level
surfaces of radial functions $u$ built so as to satisfy both an elliptic equation of the type $-\Delta u=f(u)$ on the whole $\R^n$ and the
eikonal equation $|\nabla u| =c$ on the complement of a ball.
Such construction can be adapted to treat also the case of a partially
overdetermined problem similar to (\ref{eq:over2}), but stated on an exterior domain (see Section \ref{ssect:ext2}).

\section{Counterexamples using shape optimization}\label{sect:opti}

In this section we use shape optimization in order to prove the following.

\begin{theorem}\label{optshape}
There exists an open starshaped planar domain $\Omega\subset(-1,1)^2$, different from a disk, such that, for a nonempty connected analytic subset $\Gamma$ of $\partial\Omega$,  the problem
\begin{equation}\label{eqcontre}
\left\{
\begin{array}{cclcc}
-\Delta u&=&1&in&\Omega\\
u&=&0&on&\partial\Omega\\
u_\nu&=&-c&on&\Gamma,
\end{array}
\right.
\end{equation}
admits a solution.
\end{theorem}
\begin{remark}
 Note that a nonempty analytic subset $\Gamma$ of $\partial\Omega$ is relatively open in $\partial\Omega$.
\end{remark}

The interest of this negative result should be considered in the light of the following extension of Serrin's result proved in \cite{FG}:

\begin{proposition}\label{th:opti}
Let $\Omega$ be open and bounded with $\partial\Omega$ connected. Let $\Gamma\subset\partial\Omega$ nonempty and (relatively) open. Assume there exists an open set $\widetilde{\Omega}$ with a connected {\em analytic} boundary containing $\Gamma$.
If there exists a solution $u$ of $(\ref{eq:over1})$ with $f$ analytic,
then $\Omega=\widetilde{\Omega}$, $\Omega$ is a ball, and $u$ is radially symmetric.
\end{proposition}

In particular, Proposition \ref{th:opti} implies that the analytic piece $\Gamma$ of the boundary of $\Omega$ found in Theorem \ref{optshape} cannot
be continued into a globally analytic closed ``curve'' (namely the boundary of another open set $\widetilde{\Omega}$). In the counterexample provided here, $\partial\Omega$ is piecewise analytic and globally at most $\mathcal{C}^{1,\frac{1}{2}}$ as analyzed in \cite{L}.
\par

\begin{proofx}{Theorem \ref{optshape}} Let $D=(-1,1)^2$ and $\alpha\in(\pi,4)$. We will construct $\Omega$ as an optimal set for the shape minimization problem (\ref{eq:opti}).

{From} \cite[Theorem 2.4.6]{BD} (see also \cite{henrotpierre}), we know there exists a quasi-open optimal set $\Omega^*$ which solves problem (\ref{eq:opti}).
In view of \cite[Corollary 1.2]{BHP}, $\Omega^*$ is in fact an open set. It is known that, for any open bounded set $\Omega$ (and in particular for $\Omega^*$), the functional $J$ defined in (\ref{JJ}) satisfies
$$J(\Omega)=\int_\Omega\left(\frac{1}{2}|\nabla u_\Omega|^2-u_\Omega\right)\, dx$$
where $u_\Omega$ denotes the unique solution of the homogeneous Dirichlet problem
\begin{eqnarray}\label{Dir}
\left\{
\begin{array}{ll}
-\Delta u=1\quad & \mbox{in }\Omega\\
u=0\quad & \mbox{on }\partial\Omega\ .
\end{array}
\right.
\end{eqnarray}
Since $\alpha<4$, $\Omega^*$ cannot be equal to $D$ so that the free boundary $\Gamma:=\partial\Omega^*\cap D$ is nonempty.
Moreover, by \cite[Section 5]{B}, we infer that $\Gamma$ is analytic because $f\equiv1$ is positive and analytic.
On this ``free boundary'' $\Gamma$, using the notion of shape derivative (see for instance \cite{henrotpierre}),
we classically obtain the Euler-Lagrange equation of problem (\ref{eq:opti}), namely,  (\ref{Dir}) with $\Omega=\Omega^*, u=u_{\Omega^*}$ together with
\begin{equation}\label{eq:over3}
|\nabla u_{\Omega^*}|=\Lambda>0\;\mbox{on}\;\partial\Omega^*\cap D.
\end{equation}
Since $f(u)=1>0$, the positivity of the Lagrange multiplier $\Lambda$ follows from \cite[Proposition 6.1]{B}.
By elliptic regularity, we know that there exists a unique solution $u_{\Omega^*}\in\mathcal{C}^\infty(\Omega\cup\Gamma)$ to (\ref{eq:over3}).\par
We now prove the geometric properties of solutions of (\ref{eq:opti}). First, since $\alpha>\pi$, $\Omega^*$ is not a disk. Second, we
show that $\Omega^*$ is starshaped, or at least that it may be replaced by an optimal starshaped set. To this end, we introduce $\widetilde{\Omega}:=S_XS_Y(\Omega^*)$, where $S_X$ and $S_Y$ denote the Steiner symmetrization about the
axes $OX$ and $OY$ respectively, see e.g.\ \cite{henrotpierre}, \cite{kaw}. Because of the symmetry of the square $D$ with respect to these axes,
we have $\widetilde{\Omega}\subset D$. Moreover, $|\widetilde{\Omega}|=|\Omega^*|=\alpha$ and, by well-known properties of Steiner symmetrization,
$J(\widetilde{\Omega})\le J(\Omega^*)$. Therefore, $\widetilde{\Omega}$ is also a solution of the shape optimization problem (\ref{eq:opti}) so that, as for any optimal set,
$\widetilde{\Gamma}=\partial\widetilde{\Omega}\cap D$ is smooth and $u_{\widetilde{\Omega}}$ satisfies (\ref{eqcontre}). To verify that it is starshaped, we may denote
$$\forall x\in [-1,1],\;\; A(x):=\{y\in [-1,1];\; (x,y)\in S_Y(\Omega^*)\}.$$
As a consequence of the definition of the Steiner symmetrization, we have  $[0\leq x\leq \hat{x}]\Rightarrow [A(\hat{x})\subset A(x)]$. We may also write
$$S_XS_Y(\Omega^*)=\left\{(x,y);|y|\leq \frac{1}{2}meas A(x)\right\}.$$
Since $x\in [0,1]\to meas A(x)$ is nonincreasing, we have
$$\left[|y|\leq\frac{1}{2} meas A(x), \lambda \in (0,1)\right]\;\Rightarrow\;\left[|\lambda y|\leq \frac{1}{2}meas A(x)\leq \frac{1}{2}meas A(\lambda x)\right].$$
This proves that $\widetilde{\Omega}$ is starshaped.

Therefore, $\Omega=\widetilde{\Omega}, u=u_{\widetilde{\Omega}}, c=\Lambda, \Gamma=$ any connected component of $\partial\Omega\cap D$ satisfy the statement of Theorem \ref{optshape}.
\end{proofx}

We conclude this section by mentioning some possible extensions of Theorem \ref{optshape}.

\begin{remark}The construction done in the proof of Theorem \ref{optshape} is valid in any dimension and one finds as well an optimal {\em open} set $\Omega^*\subset(-1,1)^n$ (see \cite{B} for a proof), which is different from a ball if $\alpha>\omega_n$ (the measure of the unit ball). But, the full regularity of the boundary is not proved -and probably does not hold- in any dimension. According to some recent papers (\cite{CJK,DSJ,W2,W}),
it is very likely that full regularity of the boundary may be extended to dimensions greater than 2 (up to 6? but not more?).

However, as proved in \cite{B}, the {\em reduced boundary} of this $\Omega^*$ is an analytic hypersurface and this regular part of the boundary is of positive $(n-1)$-Hausdorff measure if $\alpha<2^n$, whereas $\Omega^*$ is not a ball if $\alpha>\omega_n$.
Therefore, this also provides a (generalized) counterexample in any dimension by choosing $\Gamma$ to be this reduced boundary.
\end{remark}

\begin{remark}
In view of \cite{BL} (see also \cite[Section 3.4]{henrot}), it is possible to extend the statement of Theorem \ref{optshape} to the case when $J$ is replaced by the
shape functional $\Omega\to \lambda_1(\Omega)$, the first eigenvalue of the Laplace operator on $\Omega$ with homogeneous Dirichlet boundary conditions.
This provides one more example of an optimal domain $\Omega^*$ where $u_{\Omega^*}$, the first normalized eigenfunction,
solves (\ref{eq:over1}) with $f(u)=\lambda u$ (here, $\lambda=\lambda_1(\Omega^*)$). The proof is similar and we do not reproduce it here.
It is possible that one could go further and extend the same construction to more general sources $f(u)$, for instance of
power-type such as $f(u)=u^p$.
\end{remark}

\begin{remark} The minimal shape $\Omega^*$ for the second Dirichlet eigenvalue $\lambda_2(\Omega)$ of the Laplace operator, among all planar {\em convex} domains of given area,
is also a natural candidate for another nice counterexample. It is proved that $\Omega^*$ is not a ``stadium'' (the convex envelope of two identical tangent balls), see \cite{henrotoudet}. However, it is expected that it looks like a stadium (see \cite{henrotoudet}). If it is the case, as explained in \cite{FG}, then the first order optimality condition would lead to an overdetermined problem in which the expected overdetermined part $\Gamma$ would be the strictly convex part of $\partial\Omega^*$. The exact regularity and shape of $\Omega^*$ is still to be completely understood: see \cite[Theorems 4,6,8]{henrotoudet} and \cite{L}.
\end{remark}
\begin{remark} In the proof of Theorem \ref{optshape}, we started with some optimal shape $\Omega^*$ and adapted it so that it satisfies the required conditions. We may wonder whether {\em all optimal shapes} have the same symmetry properties. This question is related to the nontrivial question of equality case in the Steiner symmetrization, namely: is it true that $J(\Omega)=J(S_X(\Omega))$ implies that $\Omega=S_X(\Omega)$
up to a translation? We refer to \cite{CF} for this question.
\end{remark}

\section{Counterexamples via explicit construction}\label{sect:expl}

In this section we provide an explicit example of a problem of
type (\ref{eq:over2}) which admits a solution on a domain
different from a ball. We also exhibit a similar example for an
analogous exterior problem.

\subsection{A counterexample in an interior domain}\label{ssect:int2}

\begin{theorem}\label{explicit}
There exist a Lipschitz continuous and strictly
increasing function \linebreak[4]$f:\R\to(0,+\infty)$\nolinebreak[4] and $u\in\mathcal{C}^2(\overline{\Omega})$ solution of 
\begin{equation}\label{eqcontrebis}
\left\{
\begin{array}{cclcc}
-\Delta u&=&f(u)&in&\Omega\\
| \nabla u|&=&8&on&\partial\Omega\\
u&=&0&on&\Gamma,
\end{array}
\right.
\end{equation}
where $\Omega\subset\R^n$ is open, bounded, simply connected, different from a ball, with $\partial\Omega$
globally $\mathcal{C} ^ \infty$, and $\Gamma\subset\partial\Omega$ is nonempty, connected, relatively open and included in a sphere of $\R^n$.
\end{theorem}
\begin{proof}
Fix an integer $n\ge2$ and consider the function
$f:\R\to(0,+\infty)$ defined by
$$f(s)=\left\{\begin{array}{ll}
\displaystyle\frac{64(n-1)}{8-s}\quad & \mbox{if }s\le0\\
\ \\
\displaystyle4\Big[(n+2)\sqrt{s+4}-6\Big]\quad & \mbox{if }s\ge0\ .
\end{array}\right.$$
Then, $f$ is globally Lipschitz continuous and strictly increasing over $\R$.\par
Consider also the (radial) function $u$ defined on $\R^n$ by
$$u(x)=\left\{\begin{array}{ll}
(3-|x|^2)^2-4\quad & \mbox{if }|x|\le1\\
8(1-|x|)\quad & \mbox{if }|x|\ge1\ .
\end{array}\right.$$
Then, $u\in\mathcal{C}^2(\R^n)$; to see this, it suffices to write $u=u(r)$ as a function of the real variable $r=|x|$
and to note that
$$u'(r)=\left\{\begin{array}{ll}
-4r(3-r^2)\quad & \mbox{if }r\le1\\
-8\quad & \mbox{if }r\ge1\ ,
\end{array}\right.\qquad\qquad
u''(r)=\left\{\begin{array}{ll}
-12+12r^2\quad & \mbox{if }r\le1\\
0\quad & \mbox{if }r\ge1\ ,
\end{array}\right.$$
are continuous functions in $[0,\infty)$. Moreover, some computations show that $u$ satisfies
$$-\Delta u=f(u)\quad\mbox{in }\R^n\ ,\qquad u=0\quad\mbox{on }\partial B\ ,\qquad|\nabla u|=8
\quad\mbox{in }\R^n\setminus B\ ,$$
where $B$ denotes the unit ball.\par
Let $\Omega_1=\{x\in B;\, x_1<\frac{1}{2}\}$ and $D=\{x\in B;\, x_1=\frac{1}{2}\}$.
Consider a bounded domain $\Omega_2\subset\{x\in\R^n;\, x_1>\frac{1}{2}\}$ such that $D\subset\partial\Omega_2$ and
$(\partial\Omega_2\setminus \overline{D})\subset(\R^n\setminus\overline{B})$. Let
$\Omega=\Omega_1\cup D \cup\Omega_2$ (see Figure 1); for a suitable choice of $\Omega_2$ one has $\partial\Omega\in \mathcal{C}^\infty$.
Let $\Gamma=\partial\Omega_1\cap\partial\Omega$, then $u$ satisfies (\ref{eqcontrebis})
but $\Omega$ is not a ball.
\end{proof}
 \begin{figure}
\begin{center}
\includegraphics[scale = 0.2]{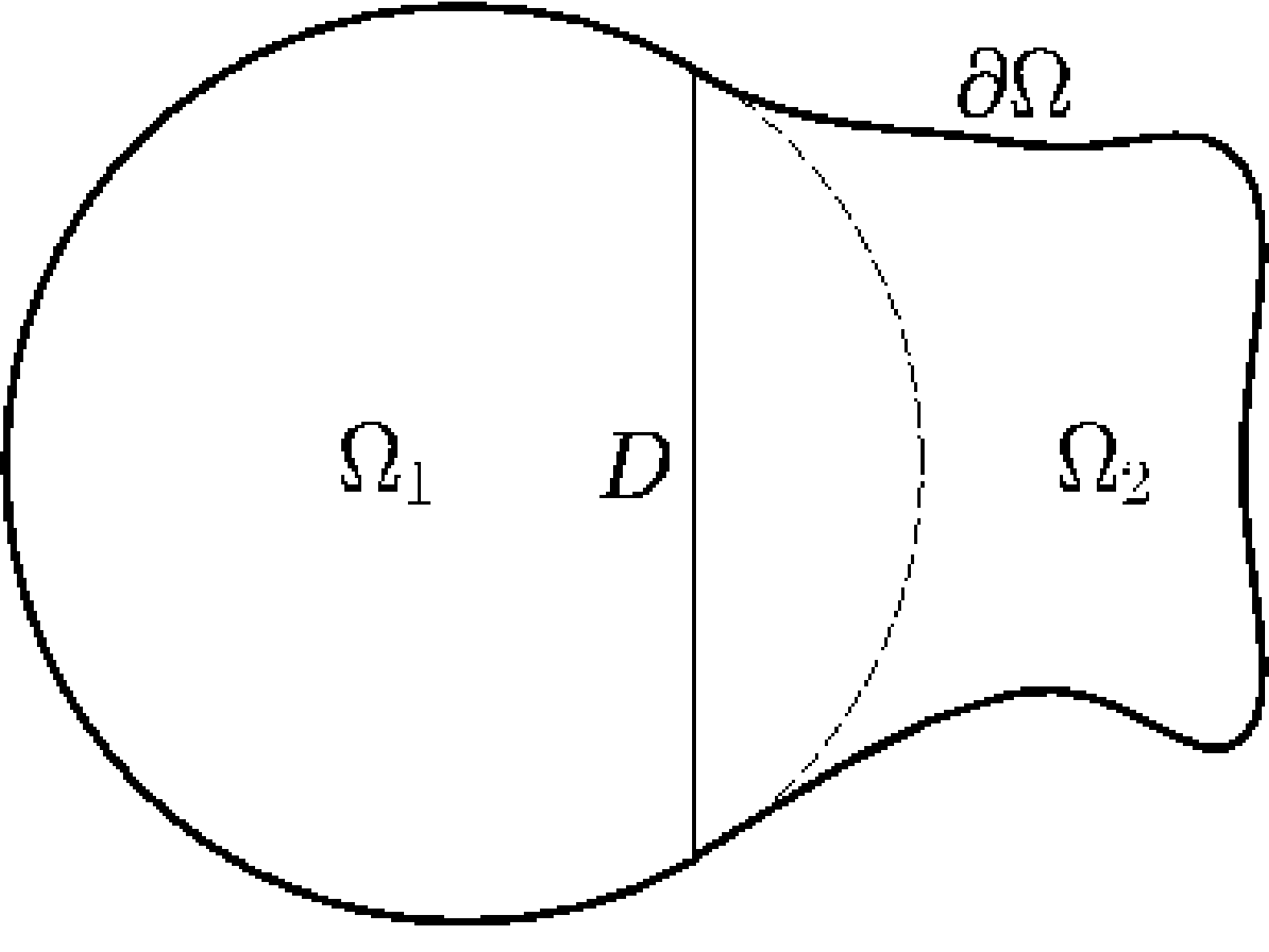}\\
 Figure 1 : domain $\Omega$ in Theorem \ref{explicit}.\\
\end{center}
\end{figure}

Theorem \ref{explicit} should be compared with the following result obtained in \cite{FG}, and similar to Proposition \ref{th:opti}:
\begin{proposition}\label{prop:analytic}
Let $\Omega$ be open and bounded with $\partial\Omega$ connected. Let $\Gamma\subset\partial\Omega$ nonempty and (relatively) open. Assume there exists an open set $\widetilde{\Omega}$ with a connected {\em analytic} boundary containing $\Gamma$.
If there exists a solution $u$ of $(\ref{eqcontrebis})$ with $f$ analytic {\bf and nonincreasing},
then $\Omega=\widetilde{\Omega}$, $\Omega$ is a ball, and $u$ is radially symmetric.
\end{proposition}
 Note in particular that: the overdetermined part $\Gamma$ in Theorem \ref{explicit} satisfy the hypothesis in Proposition \ref{prop:analytic} (analytically continuable according to the definition in \cite[Section 3.1]{FG}), but $f$ is neither analytic, nor nonincreasing.\\
 
Similarly, Theorem \ref{explicit} should also be compared with the statements (b) in Theorems 3 and 7 in \cite{FG} which gives more various sufficient conditions to obtain symmetry in overdetermined problems of type \eqref{eqcontrebis}. Again, Theorem \ref{explicit} provides an example where all these hypothesis are satisfied, except the fact that $f$ be nonincreasing.

\subsection{A counterexample in an exterior domain}\label{ssect:ext2}

\begin{theorem}\label{exterior}
There exist a Lipschitz continuous function $f: \R \to \R$, and $u\in\mathcal{C}^2(\R^n\setminus\Omega)$ solution of
\begin{equation}\label{extneumann}
\left\{\begin{array}{ll}
-\Delta u=f(u)\quad & \mbox{in }\R^n\setminus\overline{\Omega}\\
|\nabla u|=\frac{1}{2}\quad & \mbox{on }\partial\Omega\\
u=1\quad & \mbox{on }\Gamma\\
u\to0,\ |\nabla u|\to0\quad & \mbox{as }|x|\to\infty,
\end{array}\right.
\end{equation}
where $\Omega\subset\R^n$ is open, bounded, simply connected, different from a ball, with $\partial\Omega$ globally $\mathcal{C} ^ \infty$,
and $\Gamma\subset\partial\Omega$ is nonempty, connected, relatively open and included in a sphere.
\end{theorem}
\begin{proof}
  Fix an integer $n\ge2$ and consider the function
$f:\R\to\R$ defined by
$$f(s)=\left\{\begin{array}{ll}
\displaystyle\frac{n-1}{2(3-2s)}\quad & \displaystyle\mbox{if }1\le s<\frac{3}{2}\\
\ \\
\displaystyle\frac{3(n-3)}{16}\, (3-\sqrt{9-8s})^3-\frac{n-4}{16}\, (3-\sqrt{9-8s})^4\quad & \mbox{if }0<s\le1\ .
\end{array}\right.$$
Then, $f$ is globally Lipschitz continuous over $(0,\frac{3}{2})$; moreover, if $n\ge4$ then $f$ is positive and
strictly increasing.\par
Consider also the (radial) function $u$ defined on $\R^n\setminus\{0\}$ by
$$u(x)=\left\{\begin{array}{ll}
\displaystyle\frac{3-|x|}{2}\quad & \mbox{if }|x|\le1\\
\ \\
\displaystyle\frac{3}{2|x|}-\frac{1}{2|x|^2}\quad & \mbox{if }|x|\ge1\ .
\end{array}\right.$$
Then, $u\in\mathcal{C}^2(\R^n\setminus\{0\})$; to see this, it suffices to write $u=u(r)$ as a function of the real variable $r=|x|$
and to note that
$$u'(r)=\left\{\begin{array}{ll}
-\frac{1}{2}\quad & \mbox{if }0<r\le1\\
-\frac{3}{2r^2}+\frac{1}{r^3}\quad & \mbox{if }r\ge1\ ,
\end{array}\right.\qquad\qquad
u''(r)=\left\{\begin{array}{ll}
0\quad & \mbox{if }0<r\le1\\
\frac{3}{r^3}-\frac{3}{r^4}\quad & \mbox{if }r\ge1\ ,
\end{array}\right.$$
are continuous functions in $(0,\infty)$. Moreover, some computations show that $u$ satisfies
$$-\Delta u=f(u)\quad\mbox{in }\R^n\setminus\{0\}\ ,\qquad u=1\quad\mbox{on }\partial B\ ,\qquad|\nabla u|=\frac{1}{2}
\quad\mbox{in }\overline{B}\setminus\{0\}\ ,$$
where $B$ denotes the unit ball.
Take any smooth domain $\Omega\subsetneq B$ such that $0\in\Omega$ and $\{x\in\partial B;\, x_1<\frac{1}{2}\}\subset\partial\Omega$ (see figure 2).
Let $\Gamma=\partial\Omega\cap\partial B$, then $u$ satisfies (\ref{extneumann}) but $\Omega$ is not a ball.
\end{proof}
 \begin{figure}[!ht]
\begin{center}
\includegraphics[scale = 0.6]{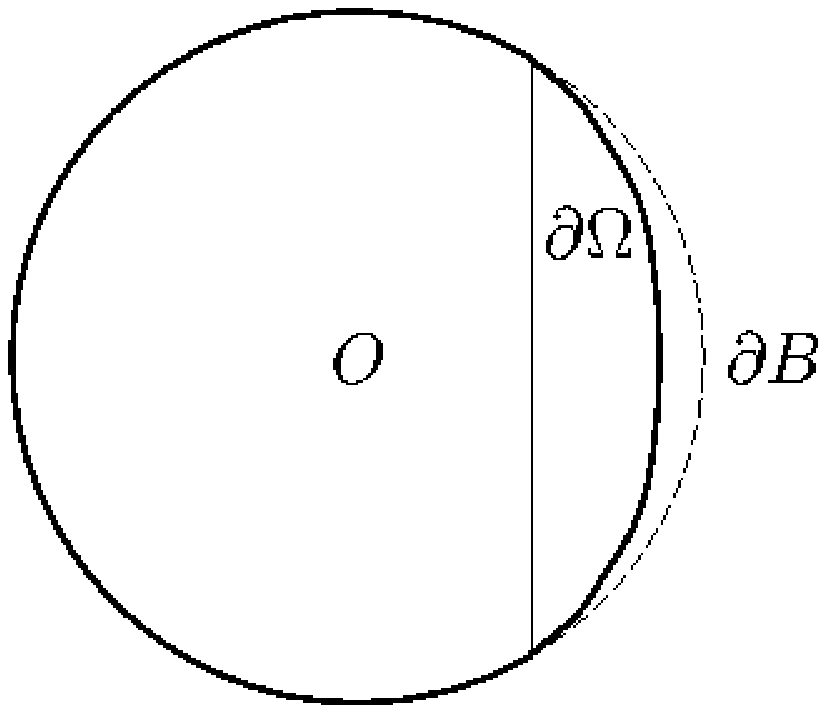}\\
 Figure 2 : domain $\Omega$ in Theorem \ref{exterior}.\\
\end{center}
\end{figure}
\begin{remark}
 Again, Theorem \ref{exterior} should be compared with the results in \cite{FG}, similarly to what we did for Theorem \ref{explicit}.
\end{remark}

\bibliographystyle{plain}

\bigskip
\noindent
\parbox[t]{.48\textwidth}{
Ilaria Fragal\`a\\
Dipartimento di Matematica\\
Politecnico di Milano \\
Piazza Leonardo da Vinci 32,\\
20133 Milano, Italy\\
ilaria.fragala@polimi.it } \hfill
\parbox[t]{.48\textwidth}{
Filippo Gazzola\\
Dipartimento di Matematica\\
Politecnico di Milano \\
Piazza Leonardo da Vinci 32,\\
20133 Milano, Italy\\
filippo.gazzola@polimi.it 
}\\
\vspace{.5cm}\\
\parbox[t]{.48\textwidth}{
Jimmy Lamboley\\
ENS Cachan Bretagne\\
IRMAR, UEB, \\
Campus de Ker Lann,\\
35170 Bruz, France\\
jimmy.lamboley@bretagne.ens-cachan.fr} \hfill
\parbox[t]{.48\textwidth}{
Michel Pierre\\
ENS Cachan Bretagne\\
IRMAR, UEB, \\
Campus de Ker Lann,\\
35170 Bruz, France\\
michel.pierre@bretagne.ens-cachan.fr 
}
\end{document}